\documentclass[11pt,reqno]{amsart}

\usepackage{amsfonts, amssymb, amscd}
\usepackage{amsmath}
\usepackage{verbatim}
\usepackage{amssymb}
\usepackage{mathrsfs}
\usepackage{graphicx}
\usepackage{cite}
\usepackage[all]{xy}
\usepackage{tikz}
\usepackage[toc,page]{appendix}
\usepackage{tikz-cd}

\usepackage{graphicx}
\usepackage{hyperref}
\usepackage{float}

\def\rk {{\operatorname{rk}}}
\def\C {{\mathbb C}}
\def\CC {{\mathbb C}}
\def\R {{\mathbb R}}
\def\Z {{\mathbb Z}}
\def\PP {{\mathbb P}}

\def\ii {{\rm i}}
\def\ee {{\rm e}}

\def\Vol {{\rm Vol}}

\theoremstyle{definition}

\newtheorem{theorem}{Theorem}[section]
\newtheorem{lemma}[theorem]{Lemma}
\newtheorem{proposition}[theorem]{Proposition}
\newtheorem{definition}[theorem]{Definition}

\newtheorem{corollary}[theorem]{Corollary}
\newtheorem{remark}[theorem]{Remark}

\numberwithin{equation}{section}
\begin{document}
\title{On Hypergeometric Duality Conjecture}

\begin{abstract}
We give an explicit formula for the duality, previously conjectured by Horja and Borisov, of two systems of GKZ hypergeometric PDEs. We prove that in the appropriate limit this duality can be identified with the inverse of the Euler characteristics pairing on cohomology of certain toric Deligne-Mumford stacks, by way of $\Gamma$-series cohomology valued solutions to the equations.  
\end{abstract}

\author{Lev Borisov}
\address{Department of Mathematics\\
Rutgers University\\
Piscataway, NJ 08854} \email{borisov@math.rutgers.edu}

\author{Zengrui Han}
\address{Department of Mathematics\\
Rutgers University\\
Piscataway, NJ 08854} \email{zh223@math.rutgers.edu}

\maketitle

\tableofcontents
\section{Introduction}\label{intro}
Let $C$ be a finite rational polyhedral cone in a lattice $N=\Z^{\rk N}$. We assume that all ray generators of $C$ lie on a primitive hyperplane $\deg(\cdot)= 1$ where $\deg:N\to\Z$ is a linear function. This data encodes an affine toric variety
$X={\rm Spec}\, \C[N^\vee \cap C^\vee]$, with the hyperplane condition equivalent to $X$ being Gorenstein, i.e. having trivial dualizing sheaf.

\smallskip
Let $\{v_i\}_{i=1}^n$ be a set of $n$ lattice points in $C$ which includes all of its ray generators, with $\deg(v_i)=1$ for all $i$.
One can construct crepant resolutions $\PP_\Sigma \to X$ by looking at subdivisions $\Sigma$ of $C$ based on triangulations that involve some of the points $v_i$. Typically, $\PP_\Sigma$ is a smooth Deligne-Mumford stack rather than a smooth variety, with the rare exception of when
all cones in $\Sigma$ are unimodular.

\smallskip
A particular case of Kawamata-Orlov $K\to D$ conjecture asserts that the derived categories of coherent sheaves on $\PP_\Sigma$ are independent of the choice of $\Sigma$. In fact, it is expected that there is an isotrivial family of triangulated categories which interpolates between the categories in question. This rather mysterious family is well understood at the level of complexified Grothendieck $K$-groups. Namely, these should correspond to solutions of a certain version of the Gel'fand-Kapranov-Zelevinsky system of hypergeometric PDEs. In fact, due to non-compactness of $X$ and $\PP_\Sigma$, there are two such systems, denoted by $\mathrm{bbGKZ}(C,0)$ and $\mathrm{bbGKZ}(C^\circ,0)$, conjecturally dual to each other \cite{BHconj}. In the appropriate limit that corresponds to the triangulation $\Sigma$, solutions to these systems can be identified with usual and compactly supported orbifold cohomology of $\PP_\Sigma$ by means of two special $\Gamma$- series.
In this paper we settle positively the duality conjecture of  \cite{BHconj}. In fact, our duality formula is simple enough to hope that it may provide hints as to how one could try to construct the aforementioned triangulated categories.

\smallskip
We will now set up the notations and review the better-behaved GKZ hypergeometric systems.
\begin{definition}\label{defGKZ}
	Consider the system of partial differential equations on the collection of functions $\{\Phi_c(x_1,\ldots,x_n)\}$ in complex variables $x_1,\ldots, x_n$, indexed by the lattice points in $C$:
	\begin{align*}
		\partial_i\Phi_c=\Phi_{c+v_i},\quad \sum_{i=1}^n\langle\mu,v_i\rangle x_i\partial_i\Phi_c+\langle\mu,c\rangle\Phi_c=0
	\end{align*}
	for all $\mu\in N^{\vee}$, $c\in C$ and $i=1,\ldots,n$. We denote this system by $\mathrm{bbGKZ}(C,0)$. Similarly by considering lattice points in the interior $C^{\circ}$ only, we can define $\mathrm{bbGKZ}(C^{\circ},0)$.
\end{definition}

\smallskip
This system gives a holonomic system of PDEs. It follows from the general theory of holonomic $D$-modules that its rank (i.e., the dimension of the solution space) is finite. For more background on this, we refer to \cite{HTT}. In contrast to the usual GKZ system where rank jumps may occur at non-generic parameters (see \cite{MMW}), it is proved in \cite{BH} that the better-behaved GKZ systems always have the expected rank which is equal to the normalized volume of the convex hull of ray generators of the cone $C$. 

\smallskip
It has been previously conjectured in \cite{BHconj} that the systems $\mathrm{bbGKZ}(C,0)$ and $\mathrm{bbGKZ}(C^{\circ},0)$ are dual to each other, in the sense that there is a pairing $\langle\cdot,\cdot\rangle$ between solutions $\Phi=(\Phi_c)$ and $\Psi=(\Psi_d)$ thereof in the form
$$
\langle \Phi,\Psi \rangle = \sum_{c,d} p_{c,d}({\bf x}) \Phi_c \Psi_d,
$$
where $p_{c,d}$ are polynomials in ${\bf x}$, with only finitely many of them nonzero. This pairing should be constant in $\bf x$ and could be viewed as the duality of the local systems of solutions. A nontrivial example of this duality has been verified in \cite{BHconj} and the
$\rk(N)=2$  case has been settled affirmatively in \cite{BHW}. 
Moreover, in certain regions of ${\bf x}$ that roughly  correspond to the complexified K\"ahler cones of $\PP_\Sigma$, one can construct solutions of 
$\mathrm{bbGKZ}(C,0)$ and $\mathrm{bbGKZ}(C^{\circ},0)$ with values in certain cohomology or $K$-theory groups of $\PP_\Sigma$.
Then it was conjectured in \cite{BHconj} that the above pairing should give (up to a constant) the inverse of a certain Euler characteristic pairing on these spaces. 
In this paper we are able to verify both statements and thus prove Conjecture 7.3 of \cite{BHconj} in full generality. 

\smallskip
Specifically,
 the following formula provides the pairing in question. Let $v\in C^\circ$ be an element in general position. 
For a subset $I\subseteq \{1,\ldots,n\}$ of size $\rk N$ we consider the cone
$\sigma_I=\sum_{i\in I} \R_{\geq 0} v_i$. We define the coefficients $\xi_{c,d,I}$ for $c+d = v_I$ as
\begin{align*}
    \xi_{c,d,I}=\left\{
    \begin{array}{ll}
         & (-1)^{\operatorname{deg}(c)}, \text{ if } \dim\sigma_{I}=\operatorname{rk}N \text{ and both } c+\varepsilon v\text{ and }d-\varepsilon v\in\sigma_{I}^{\circ} \\
         & 0, \text{ otherwise}.
    \end{array}
    \right.
\end{align*}
Here the condition has to hold for all sufficiently small $\varepsilon>0$. As usual, we denote by $\operatorname{Vol}_I$ the absolute value of the determinant of the matrix of coefficients of $v_i,~i\in I$ in a basis of $N$ (i.e., the normalized volume of $I$). We can now formulate the first result of this paper.

\medskip
\smallskip\noindent
{\bf Theorem \ref{main theorem 1}.}
For any pair of solutions $(\Phi_c)$ and $(\Psi_d)$ of $\operatorname{bbGKZ}(C,0)$ and $\operatorname{bbGKZ}(C^{\circ},0)$ respectively, the pairing
\begin{align*}
    \langle\Phi,\Psi\rangle=\sum_{c,d,I}\xi_{c,d,I}\operatorname{Vol}_I\left(\prod_{i\in I}x_i\right) \Phi_c\Psi_d
\end{align*}
is a constant. 

\smallskip
As was mentioned before, for a regular triangulation $\Sigma$ there is a description of solutions to $\operatorname{bbGKZ}(C,0)$ and $\operatorname{bbGKZ}(C^{\circ},0)$ in terms of the Gamma series $\Gamma = (\Gamma_c)$ and 
$\Gamma^\circ = (\Gamma^\circ_d)$  with values in certain orbifold cohomology spaces $H$ and $H^c$ associated to $\PP_\Sigma$, considered in \cite{BHconj}. Then the second main result of the paper is the following.

\medskip
\smallskip\noindent
{\bf Theorem \ref{main theorem 2}.}
The constant pairing $\langle \Gamma,\Gamma^\circ \rangle$ is equal up to a constant factor to the inverse of the Euler characteristic pairing $\chi(-,-):H\otimes H^c\rightarrow\C$.

\medskip
The paper is organized as follows. In Section \ref{sec.pairing} we prove the above Theorem  \ref{main theorem 1}. In Section \ref{sec.pairing-gamma} we introduce the spaces $H$ and $H^c$, the solutions $\Gamma$ and $\Gamma^c$ with values in them and compute the pairing of Theorem \ref{main theorem 1} on them. We also calculate the asymptotic behavior of the series and their pairing in the large K\"ahler limit, which is used in the next section. In Section \ref{sec.euler} we prove that this pairing is the inverse of the Euler characteristic pairing between $H$ and $H^c$. This, in particular, implies that the pairing of Theorem \ref{main theorem 1} is nondegenerate. Finally, in Section \ref{sec.generic} we explain some easy extensions of our results and state some open questions.


\section{Pairing of solutions}\label{sec.pairing}
The goal of this section is to define a pairing between the solution spaces of the better-behaved GKZ systems associated to $C$ and $C^\circ$. We first study a particular class of pairings and find a sufficient condition to make it give a constant for any pair of solutions of better-behaved GKZ systems. Then we provide a special example of this pairing, inspired by the fan displacement formula for the resolution of the diagonal in toric varieties, due to Fulton and Sturmfels \cite{FS}.

\smallskip
To state the first main result of this section, we first introduce some notations. Suppose $J$ is a subset of $\{1,2,\ldots,n\}$ with $|J|=\operatorname{rk}N+1$. We will call such subset \emph{spanning} if $\{v_i,i\in J\}$ spans $N_\R$ over $\R$. For a spanning set $J$
 there is a unique (up to multiplication by a constant factor) linear relation among the vectors $\{v_i\}_{i\in J}$
\begin{align*}
	\sum_{i\in J}a_i v_i=0.
\end{align*}
We introduce
 $\operatorname{sgn}:J\rightarrow\{0,\pm 1\}$ by $\operatorname{sgn}(j)$ being $-1$, $0$ or $1$ if $a_i$ is negative, zero or positive, respectively. This gives a decomposition $J=J_+\sqcup J_-\sqcup J_0$ of the spanning set $J$. Note that while $\operatorname{sgn}$ depends on the choice of scaling of the above linear relation, the expressions $\operatorname{sgn}(j_1)\operatorname{sgn}(j_2)$ are well-defined.
 
\smallskip
The following lemma will be used later in this section. For a subset $I\subseteq \{1,\ldots,n\}$ of size $\rk \,N$ we denote by $\Vol_I$ the normalized volume of the convex hull of the origin and $v_i,{i\in I}$. 
\begin{lemma}\label{signs}
Let $I\subset \{1,\ldots, n\}$ be such that $\{v_i, i\in I\}$ form a basis of $N_\R$. Suppose that $I$ contains $1$ and consider  $j\not\in I$. Consider the spanning set $J=I\cup \{j\}$.
Let $\mu$ denote the unique linear function  that takes value $\Vol_I$ on $v_1$ and $0$ on $v_i,i\in I\setminus 1$.
Then $\mu(v_j)= -\operatorname{sgn}(1)\operatorname{sgn}(j)\Vol_{J\setminus 1}$ for the $\operatorname{sgn}$ defined for $J$.
\end{lemma}

\begin{proof}
Up to sign, we can think of the linear function $\mu$ as taking a wedge product with $\Lambda_{i\in I\setminus 1}v_i$. 
Thus, $\mu(v_j) = \pm \Vol_{J\setminus 1}$ and we just need to determine the sign. Since $\{v_i, i\in I\}$ form a basis, the 
coefficient $a_j$ in the relation $\sum_{i\in J}a_i v_i=0$ is nonzero and we may consider it to be $1$, which ensures 
$\operatorname{sgn}(j)=1$. We apply $\mu$ to $\sum_{i\in J}a_i v_i=0$ to get 
$a_1 \Vol_I +  \mu(v_j) =0$. This implies that $a_1$ and $\mu(v_j)$ have opposite signs, and the definition of $\operatorname{sgn}(1)$ finishes the argument.
\end{proof}


Motivated by our previous work \cite{BHW}, we will look at pairings $\langle\cdot,\cdot \rangle$ that only have monomial terms $x_I = \prod_{i\in I} x_i$ for subsets $I$ of $\{1,\cdots, n\}$ of size $\operatorname{rk} N$.
The following proposition provides a sufficient condition on the pairing being a constant.
\begin{proposition}\label{equivalent condition}
Let $\{\xi_{c,d,I}\}$ be a collection of complex numbers for all $c\in C,\,d\in C^\circ,\,I\subseteq \{1,\ldots,n\}$ such 
that $c+d = \sum_{i\in I}v_i$ and $|I|=\rk \,N$. 
Suppose that 
\begin{align*}
0 = \sum_{j\in J }\operatorname{sgn}(j)  \Big(\xi_{c-v_j, d, J \setminus j} \chi(c-v_j \in C)+\xi_{c, d-v_j, J \setminus j} \chi(d-v_j \in C^\circ) \Big)
\end{align*}
holds for all $c\in C,\,d\in C^\circ$ and all spanning subsets $J\subseteq\{1,2,\ldots,n\}$ with $|J|=\operatorname{rk}N+1$ and 
$\sum_{i\in J}v_i = c+d$. Here $\chi$ denotes the characteristic function ($1$ if the statement is true and $0$ if it is false). Then
$$
\langle \Phi, \Psi\rangle =\sum_{|I|=\operatorname{rk}N} \sum_{c+d=v_I } \xi_{c,d,I} \Vol_I x_I \Phi_c \Psi_d 
$$
is a constant for any pair of solutions $(\Phi,\Psi)$.
\end{proposition}

\begin{proof}
Without loss of generality, it suffices to show that $\partial_1\langle \Phi, \Psi\rangle=0$.
We compute it as follows
\begin{equation}\label{star}
\begin{split}
	\partial_1\langle \Phi, \Psi\rangle = \sum_{\stackrel{|I|=\operatorname{rk}N,}{c+d=v_I}} &\xi_{c,d,I} \Vol_I x_I(\Phi_{c+v_1} \Psi_d + \Phi_c\Psi_{d+v_1})\\
&+\sum_{\stackrel{1\in I,|I|=\operatorname{rk}N,}{c+d=v_I}} \xi_{c,d,I} \Vol_I x_{I\setminus 1}\Phi_c\Psi_d
\end{split}
\end{equation}
and now use relations on $\Phi$ and $\Psi$ to manipulate the second sum. For each term, let $\mu$ be the linear function given by 
$$\mu(v) = v \wedge(\Lambda_{j\in I\backslash 1}v_j)$$
under the standard identification of $\Lambda^{\rk N} N_\R \cong \R$
where we choose the order of $\{v_j:j\in I\backslash 1\}$ in the wedge product such that $\mu(v_1)=\operatorname{Vol}_I$ is positive. Note that $\mu(v_j) = 0 $ for all $j\in I\setminus 1$.

\smallskip
We use $\mu(c)+\mu(d) = \mu(\sum_{i\in I}v_i) = \Vol_I$ and add appropriate multiples of equations for $\Phi_c$ and $\Psi_d$ with this $\mu$ to get 
\begin{align*}
\Phi_c\Psi_d\Vol_I &=
 -\sum_{j\notin I\setminus 1} x_j \mu(v_j)( \Phi_{c+v_j}\Psi_d+\Phi_{c}\Psi_{d+v_j}).
\end{align*}
Thus, \eqref{star} can be rewritten as 
\begin{align*}
\partial_1\langle \Phi, \Psi\rangle =& \sum_{\stackrel{|I|=\operatorname{rk}N,}{c+d=v_I}} \xi_{c,d,I} \Vol_I x_I(\Phi_{c+v_1} \Psi_d + \Phi_c\Psi_{d+v_1})
\\&
-\sum_{\stackrel{1\in I,|I|=\operatorname{rk}N,}{c+d=v_I}} \sum_{j\notin I\setminus 1} \xi_{c,d,I}\mu(v_j) x_{I_{1\to j}}( \Phi_{c+v_j}\Psi_d+\Phi_{c}\Psi_{d+v_j})
\\
=& \sum_{\stackrel{1\not\in I, |I|=\operatorname{rk}N,}{c+d=v_I}} \xi_{c,d,I} \Vol_I x_I(\Phi_{c+v_1} \Psi_d + \Phi_c\Psi_{d+v_1})
\\
&-\sum_{\stackrel{1\in I,|I|=\operatorname{rk}N,}{c+d=v_I}} \sum_{j\notin I} \xi_{c,d,I}\mu(v_j) x_{I_{1\to j}}( \Phi_{c+v_j}\Psi_d+\Phi_{c}\Psi_{d+v_j})
\end{align*}
where we canceled the terms with $1\in I$ in the first sum with $j=1$ in the second sum. Here $I_{1\rightarrow j}=I\backslash\{1\}\cup\{j\}$. Note that $\mu$ depends on the set $I$.

\medskip
Let us now compute the coefficient at $x_{\hat I}\Phi_{\hat c}\Psi_{\hat d}$ in the above expression. This coefficient  gets contributions from the first sum with $I=\hat I$ and from the second sum with $I= \hat I \cup 1\setminus j$.
We observe that $\hat I$ has size $\operatorname{rk}N$ and  does not contain $1$.  Also note that if $\Vol_{\hat I}=0$, then the coefficient is zero. Indeed, in the second sum, $\mu(v_j)=\pm \Vol_{I_{1\rightarrow j}}$.
Finally, we must have 
$$
\hat c+ \hat d = v_J, ~J = { \{1\}\sqcup \hat I}.
$$
We look at the set $J$ which we know to be spanning, since it contains $\hat I$. 
By Lemma \ref{signs}, we see that $\mu(v_j) = \operatorname{sgn}(1)\operatorname{sgn}(j) \Vol_{\hat I}$.
Therefore, the first line contributes
$$
\xi_{\hat c-v_1,\hat d,\hat I}\Vol_{\hat I}  \chi(\hat c-v_1\in C)+  \xi_{\hat c,\hat d-v_1,\hat I}\Vol_{\hat I} \chi(\hat d-v_1\in C^\circ)
$$
and the second line contributes
\begin{align*}
\sum_{j\in J\setminus 1}\operatorname{sgn}(1)\operatorname{sgn}(j)\xi_{\hat c-v_j,\hat d,J\setminus j}\operatorname{Vol}_{\hat I}  \chi(\hat c-v_j\in C) 
\\
+\sum_{j\in J\setminus 1}\operatorname{sgn}(1)\operatorname{sgn}(j)\xi_{\hat c,\hat d-v_j,J\setminus j}\Vol_{\hat I}  \chi(\hat d-v_j\in C^\circ).
\end{align*}
We observe that if $\operatorname{sgn}(1)=0$, then $\{v_i, i\in J\setminus 1\}$ do not span $N_\R$, so $\Vol_{\hat I}=0$ and the statement trivially holds. Thus we can introduce $\operatorname{sgn}(1)^2$ into the first term to have the coefficient
at $x_{\hat I}\Phi_{\hat c}\Psi_{\hat d}$ equal
\begin{align*}
\operatorname{sgn}(1) \operatorname{Vol}_{\hat I} \sum_{j\in J}\operatorname{sgn}(j)\Big(\xi_{\hat c-v_j,\hat d,J\setminus j} \chi(\hat c-v_j\in C) 
+\xi_{\hat c,\hat d-v_j,J\setminus j}\chi(\hat d-v_j\in C^\circ)\Big),
\end{align*}
and the claim follows.
\end{proof}

\begin{remark}
After some sign changes, one can rephrase the condition of Proposition \ref{equivalent condition} as $d\xi=0$ for an 
appropriate element $ \xi\in\C[C]\otimes \C[C^\circ] \otimes \Lambda^{\rk N}(\oplus_{i=1}^n \C e_i)$ with the differential
$$
d = \sum_{i=1}^n [v_i]\otimes 1\otimes (e_i \wedge)+ \sum_{j=1}^n 1\otimes [v_j]\otimes (e_j \wedge)
$$
on  $ \xi\in\C[C]\otimes \C[C^\circ] \otimes \Lambda^{\bullet}(\oplus_{i=1}^n \C e_i)$.
We do not pursue this direction further in the paper.
\end{remark}

\smallskip
Now we give an explicit formula of the pairing $\langle-,-\rangle$ between solutions of the better-behaved GKZ systems $\mathrm{bbGKZ}(C,0)$ and $\mathrm{bbGKZ}(C^{\circ},0)$. We prove that $\langle\Phi,\Psi\rangle$ is a constant for any pair of solutions $\Phi$ and $\Psi$ by using Proposition \ref{equivalent condition}.

\smallskip
Fix a choice of a generic vector $v\in C^{\circ}$. For a set $I$ of size $\rk N$ we consider the cone
$\sigma_I=\sum_{i\in I} \R_{\geq 0} v_i$. We define the coefficients $\xi_{c,d,I}$ for $c+d = v_I$ as
\begin{align}\label{beta-good}
    \xi_{c,d,I}=\left\{
    \begin{array}{ll}
         & (-1)^{\operatorname{deg}(c)}, \text{ if } \dim\sigma_{I}=\operatorname{rk}N \text{ and both } c+\varepsilon v\text{ and }d-\varepsilon v\in\sigma_{I}^{\circ} \\
         & 0, \text{ otherwise}.
    \end{array}
    \right.
\end{align}
Here the condition has to hold for all sufficiently small $\varepsilon>0$. It is clear that $\xi$ is well-defined as long as the vector $v$ is chosen sufficiently generic. Note that $\xi_{c,d,I}\not=0$ implies that both $c$ and $d$ lie in the maximum-dimensional cone $\sigma_I$ (but not necessarily in its interior). 

\smallskip
We are now ready to tackle the main result of this section.
\begin{theorem}\label{main theorem 1}
For any pair of solutions $(\Phi_c)$ and $(\Psi_d)$ of $\operatorname{bbGKZ}(C,0)$ and $\operatorname{bbGKZ}(C^{\circ},0)$ respectively, the pairing
\begin{align*}
    \langle\Phi,\Psi\rangle=\sum_{c,d,I}\xi_{c,d,I}\operatorname{Vol}_I\left(\prod_{i\in I}x_i\right) \Phi_c\Psi_d
\end{align*}
is a constant.
\end{theorem}

\begin{proof}
We prove this theorem by showing that these coefficients $\xi_{c,d,I}$ satisfy the conditions in Proposition \ref{equivalent condition}, namely 
\begin{align*}
    0=\sum_{j \in J} \xi_{c-v_j,d, J \backslash j}\operatorname{sgn}(j) \chi\left(c-v_j \in C\right)+\sum_{j \in J} \xi_{c,d-v_j, J \backslash j}\operatorname{sgn}(j) \chi\left(d-v_j \in C^{\circ}\right) 
\end{align*}
for all spanning subsets $J\subseteq\{1,2,\cdots,n\}$ with $|J|=\operatorname{rk}N+1$, and all $c\in C,~d\in C^\circ$ with $c+d=\sum_{i\in J}v_i$. 

\smallskip
We first observe that the conditions $c-v_j \in C$ and $d-v_j \in C^{\circ}$ in the equations above are redundant.
Indeed, to ensure that $\xi_{c-v_j,d, J \backslash j}\not=0$ we must have $c-v_j+\varepsilon v\in\sigma_{J\backslash j}^{\circ}$,
which implies that $c-v_j\in \sigma_{J\backslash j}\subseteq C$.
For the second term, to ensure that $\xi_{c,d-v_j, J \backslash j}\not=0$ we must have $d-v_j-\varepsilon v\in\sigma_{J\backslash j}^{\circ}\subseteq C^{\circ}$ (since $J\backslash j$ is a maximal cone), which implies $d\in C^\circ$.
Thus, it suffices to consider the equations
\begin{align}\label{above}
    0=\sum_{j \in J_+\sqcup J_-} \xi_{c-v_j,d, J \backslash j}\operatorname{sgn}(j) +\sum_{j \in J_+\sqcup J_-} \xi_{c,d-v_j, J \backslash j}\operatorname{sgn}(j)
\end{align}
for $\xi$ defined in \eqref{beta-good}.
The nonzero terms occur for the indices $j$ such that both $c-v_j+\varepsilon v$ and $d-\varepsilon v$ lie in $\sigma_{J\backslash j}^{\circ}$, or both $c+\varepsilon v$ and $d-v_j-\varepsilon v$ lie in $\sigma_{J\backslash j}^{\circ}$.

\smallskip
We consider the equation in the variables $a_i$
\begin{align*}
    \sum_{i\in J}a_i v_i=c+\varepsilon v.
\end{align*}
The solution set to this equation is an affine line $l_{c+\varepsilon v}$ in the space $\R^{\rk N+1}$. A contribution to the first term of \eqref{above} happens when there is a point on $l_{c+\varepsilon v}$ with $a_j=1$ and all other $a_i$ lie in $(0,1)$ due to the definition of the coefficient $\xi$. Similarly, a contribution to the second term happens for $a_j=0$ and all other $a_i$ lie in $(0,1)$.

\smallskip
Recall from Lemma \ref{signs} that we have a decomposition $J=J_+\sqcup J_-\sqcup J_0$. For $i\in J_0$, the value of $a_i$ on the line $l_{c+\varepsilon v}$ is constant. Since $v$ is generic, we may assume it to be non-integer. Thus, it either prohibits any contributions to \eqref{above} (if $a_i\not\in (0,1)$) or provides no restrictions. Therefore, we may now assume that the latter happens for all $i\in J_0$.

\smallskip
The key idea of the proof is to consider the line segments 
\begin{align*}
    S_i=l_{c+\varepsilon v}\cap \{0\leq a_i\leq 1\}
\end{align*}
on $l_{c+\varepsilon v}$ for all $i\in J_+\sqcup J_-$. The nonzero contributions to \eqref{above} happen exactly for the endpoints of a line segment $S_j$ that lie strictly inside all other segments. The assumption that $\varepsilon v$ is generic implies that the endpoints of different $S_i$ do not coincide. Indeed, if it were the case, then $c+\varepsilon v$ would lie in a shift of the span of ${\rk N}-1$ of $v$-s by a lattice element, and we may assure that this does not happen. Consider now $S = \bigcap_{i\in  J_+\sqcup J_-}S_i$. If $S$ is empty then there are no contribution, since this point would not lie in the interior of other $S_i$. So it suffices to consider the case when $S$ is a segment $[p,q]$. It is clear that the only points that could contribute to \eqref{above} are $p$ and $q$. In particular, there are at most two nonzero terms in \eqref{above}. We will show that they always cancel each other.

\smallskip
We also note that the orientation of the segment $S_i$ (i.e., the direction in which the parameter $a_i$ increases) on the line $l_{c+\varepsilon v}$ is determined by $\operatorname{sgn}(i)$, since 
the vector along the line is given by the nontrivial linear relation on $v_{k,k\in J}$. If both $p$ and $q$ are the $a_i=1$ and $a_j=1$ ends of the segments $S_i$ and $S_j$, then the segments must have opposite orientations on $l_{c+\varepsilon v}$ (since they both should point towards the other point). This means that $\operatorname{sgn}(i)=-\operatorname{sgn}(j)$ and the two terms of \eqref{above} cancel.  Similarly, they cancel if $p$ and $q$ are the $a_i=0$ and $a_j=0$ ends of $S_i$ and $S_j$.

\smallskip
Now suppose that $p$ and $q$ correspond to $a_i=0$ and $a_j=1$ ends of $S_i$ and $S_j$ (in this case it is possible to have $i=j$). In this case the two segments must have the same orientation, and then the factor $(-1)^{\deg c}$ in the definition of $\xi$ ensures that the two terms cancel each other.
\end{proof}

\begin{remark}
As $v$ varies, we get a finite number of different formulas for the pairing. It is also possible to take a more uniform choice of the pairing by integrating over $v$ of degree $1$ (ignoring the contributions of measure zero set of nongeneric $v$). However, there does not appear to be any advantage in doing so. We will later see that the pairing is in fact independent of the choice of $v$.
\end{remark}

\section{Pairing of the Gamma series}\label{sec.pairing-gamma}
In this section we compute the pairing from the previous one on the cohomology-valued solutions to the better-behaved GKZ systems provided by the $\Gamma$ series. We will show in the next section that the result is the dual of the intersection pairing which provides the proof of Conjecture 7.3 from \cite{BHconj}. 

\smallskip
We consider a regular triangulation $\Sigma$ of the cone $C$ whose vertices are among these vectors $\{v_i\}_{i=1}^n$ and its corresponding toric Deligne-Mumford stack $\mathbb{P}_{\Sigma}$. 
\begin{remark}
It will be convenient for us to abuse notation and  denote by $I$ both a subset of $\{1,\ldots,n\}$ and the corresponding cone $\sum_{i\in I} \R_{\geq 0}v_i$. Similarly, $\Sigma$ denotes both a simplicial complex on  $\{1,\ldots,n\}$ and the corresponding simplicial fan in $N_\R$ which refines $C$ and its faces.
\end{remark}

\begin{definition}
	For each cone $\sigma\in\Sigma$ we define $\operatorname{Box}(\sigma)$ to be the set of lattice points $\gamma$ which can be written as $\gamma=\sum_{i\in\sigma}\gamma_i v_i$ with $0\leq\gamma_i<1$. We denote the union of all $\operatorname{Box}(\sigma)$ by $\operatorname{Box}(\Sigma)$. To each element $\gamma\in\operatorname{Box}(\Sigma)$ we associate a  \emph{twisted sector} of $\mathbb{P}_{\Sigma}$ corresponding to the minimal cone $\sigma(\gamma)$ in $\Sigma$ containing $\gamma$. 
We define the dual of a twisted sector $\gamma=\sum\gamma_i v_i$ by
\begin{align*}
	\gamma^{\vee}=\sum_{\gamma_i\not=0}(1-\gamma_i)v_i.
\end{align*}
or equivalently, the unique element in $\operatorname{Box}(\sigma(\gamma))$ that satisfies
\begin{align*}
	\gamma^{\vee}=-\gamma\hskip -3pt \mod \sum_{i\in\sigma}\Z v_i
\end{align*}
\end{definition}

\begin{remark}
The dual of $\gamma=0$ is itself. Clearly, we have $\sigma(\gamma)=\sigma(\gamma^{\vee})$ and $(\gamma^\vee)^\vee=\gamma$.
\end{remark}

\smallskip
Twisted sectors are themselves smooth toric DM stacks and the following propositions describe a  Stanley-Reisner type presentation of the spaces of cohomology and cohomology with compact support of their coarse moduli spaces, see \cite{BHconj}.

\begin{proposition}\label{coh}
As usual, $\operatorname{Star}(\sigma(\gamma))$ denotes the set of cones in $\Sigma$ that contain $\sigma(\gamma)$.
	Cohomology space $H_{\gamma}$ of the twisted sector $\gamma$ is naturally isomorphic to the quotient of the polynomial ring $\CC[D_i:i\in\operatorname{Star}(\sigma(\gamma))\backslash\sigma(\gamma)]$ by the ideal generated by the relations
	\begin{align*}
		\prod_{j\in J}D_j,\ J\not\in\operatorname{Star}(\sigma(\gamma)),\quad\text{and }\sum_{i\in\operatorname{Star}(\sigma(\gamma))\backslash\sigma(\gamma)}\mu(v_i)D_i,\ \mu\in\operatorname{Ann}(v_i,i\in\sigma(\gamma)).
	\end{align*}
We can also view $H_\gamma$ as a module over the polynomial ring $\CC[D_1,\ldots, D_n]$ by declaring $D_i=0$ for $i\not\in \operatorname{Star}(\sigma(\gamma))$ and solving (uniquely) for 
$D_i, i\in \sigma(\gamma)$ to satisfy the linear relations $\sum_{i=1}^n \mu(v_i)D_i=0$ for all $\mu\in N^\vee$.
\end{proposition}

\begin{proposition}\label{cohcmp}
	Cohomology space with compact support $H_{\gamma}^c$ (viewed as a module over $H_\gamma$) is generated by $F_I$ for $I\in\operatorname{Star}(\sigma(\gamma))$ such that $\sigma_I^{\circ}\subseteq C^{\circ}$ with relations
	\begin{align*}
		&\quad D_i F_I-F_{I\cup\{i\}}\text{ for }i\not\in I,I\cup\{i\}\in\operatorname{Star}(\sigma(\gamma))\\
		&\text{and }D_i F_I\text{ for }i\not\in I,I\cup\{i\}\not\in\operatorname{Star}(\sigma(\gamma))
	\end{align*}
Similarly, it is given a structure of a module over $\CC[D_1,\ldots, D_n]$.
\end{proposition}

\begin{definition}
	The orbifold cohomology $H$ of the smooth toric DM stack $\mathbb{P}_{\Sigma}$ is defined as the direct sum $\bigoplus_{\gamma} H_{\gamma}$ over all twisted sectors. Similarly, the orbifold cohomology with compact support $H^c$ is defined as $\bigoplus_{\gamma} H_{\gamma}^c$. We denote by $1_\gamma$ the generator of $H_\gamma$.
\end{definition}

\smallskip
There is a natural perfect pairing between $H$ and $H^c$ called \emph{Euler characteristic pairing}. Its origin is the eponymous pairing on certain Grothendick $K$-groups, which is then translated to the cohomology via the Chern character, see \cite{BHconj}. We will not be using the original definition, but rather the following formula for the Euler characteristic pairing, which is proved in \cite{BHW}.

\begin{proposition}\label{def-eulerpairing}
The Euler characteristic pairing $\chi:H\otimes H^c\rightarrow\C$ on the toric DM stack $\mathbb{P}_{\Sigma}$ is given by
	\begin{align*}
		\chi(a,b)=\chi(\oplus_{\gamma}a_{\gamma},\oplus_{\gamma}b_{\gamma})=\sum_{\gamma}\frac{1}{|\operatorname{Box}(\sigma(\gamma))|}\int_{\gamma^{\vee}}\operatorname{Td}(\gamma^{\vee})a_{\gamma}^*b_{\gamma^{\vee}}
	\end{align*}
	Here $*:H\rightarrow H$ is the duality map given by $(1_{\gamma})^*=1_{\gamma^{\vee}}$ and $(D_i)^*=-D_i$, and $\operatorname{Td}(\gamma)$ is the Todd class of the twisted sector $\gamma$ which is defined as
	\begin{align*}
		\operatorname{Td}(\gamma)=\frac{\prod_{i\in\operatorname{Star}\sigma(\gamma)\backslash\sigma(\gamma)}D_i}{\prod_{i\in\operatorname{Star}\sigma(\gamma)}(1-e^{-D_i})}.
	\end{align*}
	\smallskip
The linear function $\int\colon H_{\gamma}^c\rightarrow\C$ takes values $\frac{1}{\Vol_{\overline{I}}}$ on each generator $F_I$, where $\Vol_{\overline{I}}$ denotes the volume of the cone $\overline{\sigma_{I}}$ in the quotient fan $\Sigma/\sigma(\gamma)$.  It takes value zero on all elements 
of $H_{\gamma}^c$ of lower degree. 
\end{proposition}

\smallskip
Let $\Sigma$ be a regular (=projective) subdivision of $C$ based on some of the $v_i$. Let $\psi_i$ be the real numbers such that 
$\Sigma$ reads off the lower boundary of the convex hull of the origin and $\{(v_i,\psi_i), 1\leq i\leq n\}$ in $N_\R\oplus \R$. We assume that $\psi_i$ are generic so this convex hull is simplicial.
We denote by $\psi$ the strictly convex piecewise linear function on $C$ whose graph is the aforementioned lower boundary.
It takes values $\psi_i$ on all $v_i$ which generate rays in $\Sigma$ and has lower values than $\psi_i$ on other $v_i$.
Its key property is that for any finite collection $w_i\in C$ and $\alpha_i\in \R_{>0}$ there holds
$$
\psi(\sum_i \alpha_i w_i) \leq \sum_i \alpha_i \psi(w_i)
$$
with equality if and only if there exists a cone in $\Sigma$ which contains all of the $w_i$. 

\smallskip
Recall from \cite{BHconj} the following solution to the equations $\mathrm{bbGKZ}(C,0)$ with values in $H=\bigoplus_\gamma H_\gamma$. We define
\begin{align}\label{Gamma}
\Gamma_c(x_1,\ldots,x_n) =\bigoplus_{\gamma} \sum_{l\in L_{c,\gamma}} \prod_{i=1}^n 
\frac {x_i^{l_i+\frac {D_i}{2\pi \ii }}}{\Gamma(1+l_i+\frac {D_i}{2\pi \ii })}
\end{align}
where the direct sum is taken over twisted sectors $\gamma = \sum_{j\in \sigma(\gamma)} \gamma_j v_j$ and the set $L_{c,\gamma}$ is the set of solutions to $\sum_{i=1}^n l_i v_i = -c$ with $l_i-\gamma_i \in \Z$ for all $i$. The numerator is defined by picking a branch of $\log(x_i)$.

\smallskip
We will first prove that for each $c\in C\cap N$ the series for $\Gamma$ converges absolutely and uniformly on compacts for ${\bf x}$ such that the $(-\log |x_i|)$ are in an appropriate shift of the cone of values on $v_i$ of convex $\Sigma$-piecewise linear functions. 
The proof was skipped in  \cite{BHconj} because it is essentially the same as that in  \cite{MellinBarnes}, but we will present it here, both for completeness and to facilitate arguments about the asymptotic behavior of $\Gamma_c$.

\begin{proposition}\label{convgamma}
We denote by $C_\Sigma$ the cone of the secondary fan that corresponds to $\Sigma$, i.e. the cone of $(\psi_i)\in \R^n$ that give rise to $\Sigma$.
For each $c\in C\cap N$ there exists $\hat \psi\in \R^n$ such that 
the series \eqref{Gamma} converges absolutely and uniformly on compacts in the region of $\C^n$
\begin{align}\label{region}
\{(-\log |x_1|,\ldots,-\log|x_n|)\in 
\hat\psi + C_\Sigma,~\arg({\bf x}) \in (-\pi ,\pi)^n\}.
\end{align}
\end{proposition}

\begin{proof}
An immediate observation is that we can ignore the factor
$$\prod_{i=1}^n x_i^{\frac {D_i}{2\pi \ii}} = \prod_{i=1}^n \ee^{\frac {D_i \log x_i}{2\pi \ii} }$$
because it does not depend on $l$ and is bounded on compacts in the region \eqref{region}.

\smallskip
It suffices to understand what happens for a fixed $\gamma$. Note that while the summation takes place over an affine lattice $L_{c,\gamma}$, the nonzero contributions only occur for $(l_1,\ldots,l_n)$ such that the set
$$I(l)=\{i, l_i\in \Z_{<0}\}\sqcup \sigma(\gamma)$$
is a cone $\sigma$ in $\Sigma$, because each $l_i\in \Z_{<0}$ contributes a factor $D_i$ due to a pole of $\Gamma$ at a nonpositive integer. Consequently, it suffices to bound the summation over the subset $L_{c,\gamma,\sigma}$  of $L_{c,\gamma}$ with the additional property that the above defined $I(l)$ is a subset of  some fixed maximum-dimensional cone $\sigma$ of $\Sigma$ that contains $\sigma(\gamma)$.
For any such $l\in L_{c,\gamma,\sigma}$ we have 
$$
\sum_{i,l_i<0} (-l_i) v_i= \sum_{i, l_i\geq 0} l_i v_i + c.
$$
Let us denote by $\psi$ the $\Sigma$-piecewise linear convex function that corresponds to $(-\hat \psi_i - \log |x_i|)$ by the assumption on ${\bf x}$.  Since the $v_i$ on the left hand side of the above equation lie in $\sigma\in \Sigma$, we have   
\begin{align*}
\sum_{i,l_i<0}(-l_i) (-\hat \psi_i - \log |x_i|) &= \psi(\sum_{i,l_i<0} (-l_i) v_i) \leq \sum_{i, l_i\geq 0}
 l_i \psi(v_i) + \psi(c)\\
 &=\sum_{i,l_i\geq 0}l_i ( -\hat \psi_i - \log |x_i|) +\psi(c)
\end{align*}
and therefore
\begin{align}\label{convexbound}
\sum_{i=1}^n l_i \log|x_i| \leq -\sum_{i=1}^n l_i \hat\psi_i + \psi(c).
\end{align}
This leads to an upper bound
\begin{align}\label{bound}
\Big\vert
\prod_{i=1}x_i^{l_i} \Big\vert \leq \ee^{\psi(c)} \ee^{- \sum_{i=1}^n l_i \hat\psi_i }.
\end{align}

\smallskip
Crucially, since 
all $v_i$ have degree $1$, we see that $\sum_i l_i = -\deg c$. Thus, we can apply the key estimate of \cite[Lemma A.4]{MellinBarnes} which states that for any $\delta>0$ and any collection of real numbers $a_i, b_i$ for $i=1,\ldots,n$ with 
$$
|\sum_{i}a_i |\leq \delta,~~\sum_i |b_i| \leq\delta
$$
there exists a constant $A$ such that 
$$
\Big\vert \prod_{i=1}^n \frac 1{\Gamma(a_i+{\rm i} b_i)}\Big\vert \leq A (4n)^{\sum_{i=1}^n |a_i|}.
$$
By the Cauchy's formula for partial derivatives, this implies an upper bound of the form $A_1 (A_2)^{\sum_{i=1}^n |l_i|}$ on the coefficients on all monomials in $D_i$ of bounded degree of the function
$$ \prod_{i=1}^n \frac 1{\Gamma(1 + l_i +\frac {D_i}{2\pi{\rm i}})}.
$$
Together with \eqref{bound}, we conclude that in any Euclidean norm on $H_\gamma$ the absolute value of each term of the series is bounded by
\begin{align}\label{A3A2bound}
\Big\vert\prod_{i=1}^n 
\frac {x_i^{l_i}}{\Gamma(1+l_i+\frac {D_i}{2\pi \ii })}\Big\vert
\leq  A_1 (A_2)^{\sum_{i=1}^n |l_i|}\Big\vert\prod_{i=1}x_i^{l_i} \Big\vert 
\leq
  A_3 (A_2)^{\sum_{i=1}^n |l_i|} 
\ee^{
- \sum_{i=1}^n l_i \hat\psi_i 
}.
\end{align}

\smallskip
We observe that the set $L_{c,\gamma,\sigma}$ is the set of lattice points in a shift of a (lower-dimensional) polyhedral cone $C_{\sigma}$ in 
$\R^n$ given by the equality $\sum_{i=1}^n l_i v_i = 0$ and inequalities $l_i\geq 0$ for all $i\not\in\sigma$. We may assume $\hat\psi$ to give a strictly $\Sigma$-convex function. It then follows that for any ray generator $l$ of $C_{\sigma}$ there holds 
$$
\sum_i l_i \hat\psi_i  > 0.
$$
Indeed, by convexity for $\hat\psi$ for any $l\in C_{\sigma}$ we have the inequality $\sum_i l_i \hat\psi_i  \geq 0$ (the proof is the same as that of \eqref{convexbound}) which holds even if $\hat\psi$ is deformed slightly, so it can only be equality for $l=0$.
As a consequence, there is a constant $r$ such that 
$$
\sum_{i=1}^n |l_i| \leq r (\sum_i l_i \hat\psi_i )
$$
on $C_{\sigma}$. 

\smallskip
Therefore, we can replace $\hat\psi$ by a large enough multiple of itself and use \eqref{A3A2bound} to get on any compact subset of the region \eqref{region}
\begin{align*}
\Big\vert\prod_{i=1}^n 
\frac {x_i^{l_i}}{\Gamma(1+l_i+\frac {D_i}{2\pi \ii })}\Big\vert
\leq
  A_4 
\ee^{-A_5 \sum_{i=1}^n l_i \hat\psi_i}
\end{align*}
for some $A_5>0$. Since the number of terms in $L_{c,\gamma,\sigma}$ with 
 $\sum_{i=1}^n l_i \hat\psi_i \in [m,m+1)$ is bounded by a polynomial in $m$, we get the desired convergence.
\end{proof}

\smallskip
There is a similarly defined  $\Gamma$-series solution $\Gamma^\circ$ of $\mathrm{bbGKZ}(C^\circ,0)$, with values in $H^c=\bigoplus_\gamma H_\gamma^c$. We define
\begin{align*}
\Gamma_c^{\circ}(x_1,\ldots,x_n) =\bigoplus_{\gamma} \sum_{l\in L_{c,\gamma}} \prod_{i=1}^n 
\frac {x_i^{l_i+\frac {D_i}{2\pi \ii }}}{\Gamma(1+l_i+\frac {D_i}{2\pi \ii })}\left(\prod_{i\in\sigma}D_i^{-1}\right)F_{\sigma}
\end{align*}
where $\sigma$ is the set of $i$ with $l_i\in\Z_{<0}$.

\begin{proposition}
The series $\Gamma^\circ$ converges uniformly on compacts in the region \eqref{region} for an appropriate choice of $\hat\psi$.
\end{proposition}

\begin{proof}
The idea of the proof are the same as that of Proposition \ref{convgamma} and we leave the details to the reader.
\end{proof}

\smallskip
Our next goal is to understand the asymptotic behavior of 
$$\Gamma_c(t^{-\psi(v_1)} x_1, \ldots, t^{-\psi(v_n)}x_n)$$
for real $t\to +\infty$. We can assume $x_i$ to be generic nonzero complex numbers, so that for large enough $t$ we fall within the range of convergence of $\Gamma$. 

\smallskip
For each $c$ we consider the minimum cone $\sigma(c)$ of $\Sigma$  that contains $c$. We have $c=\sum_{j\in \sigma(c)} c_j v_j$. It defines a twisted sector $\gamma(c) = \sum_{j\in \sigma(c)} \{c_j\}v_j$.  We also consider the dual twisted sector 
$\gamma^\vee(c)=\sum_{j\in \sigma(c), c_j\not\in \Z}(1-\{c_j\})v_j$. There is a special element 
\begin{align}\label{leading}
-c = \sum_{i\in I}(-c_i) v_i 
\end{align}
in $L_{c,\gamma^\vee(c)}$.

\begin{lemma}\label{lead-c}
As $t\to +\infty$, we have for $c\in C\cap N$ and $\gamma\neq  \gamma^\vee(c)$ the $\gamma$ summand of 
$\Gamma_c(t^{-\psi(v_1)} x_1, \ldots, t^{-\psi(v_n)}x_n)$ is $o(t^{\psi(c)})$. For $\gamma=\gamma^\vee(c)$ we have
\begin{align*}
\Gamma_c(t^{-\psi(v_1)} x_1, \ldots)
=
t^{\psi(c)} \prod_{i=1}^n {\rm e}^{\frac {D_i}{2\pi{\rm i}}( \log x_i - \psi(v_i) \log t)}
\prod_{i=1}^{n}\frac {x_i^{-c_i}}{\Gamma(1-c_i +\frac {D_i}{2\pi{\rm i}} )}
(1+o(1)).
\end{align*}
\end{lemma}

\begin{proof}
Let $\gamma = \sum_{j\in \sigma(\gamma) }\gamma_j v_j$.
Let $(l_i)$ be an element of $L_{c,\gamma}$. The contribution to $\Gamma_c(t^{-\psi(v_1)} x_1, \ldots)$ is only nonzero if the set of
$i$ for which $l_i\in \Z_{<0}$ together with $\sigma(\gamma)$ is a cone in $\Sigma$. Consequently, $i$ for which $l_i$ are negative lie in a cone of $\Sigma$.  
Therefore,
\begin{align}\label{inpsi}
 \sum_{l_i<0}(-l_i) \psi(v_i) = \psi(\sum_{l_i<0}(-l_i) v_i) = \psi(c+\sum_{l_i>0}l_i v_i) \leq   \sum_{l_i>0} l_i\psi(v_i) + \psi(c),
\end{align}
which implies
\begin{align}\label{inequality}
	-\sum_{i=1}^n l_i \psi(v_i) \leq \psi(c).
\end{align}
Now notice that the equality in \eqref{inequality} holds if and only if the minimal cone of $\sum_{l_i<0}(-l_i) v_i$ is a cone in $\Sigma$ which contains $c$ and all $v_i$ with $l_i> 0$. This cone would then contain $c$ and all $v_i$ for which $l_i\neq 0$. This means that $l_i = -c_i$, which implies that $\gamma=\gamma^\vee(c)$. This gives the claimed asymptotic contribution.

\medskip
It is not enough to bound the asymptotic behavior of each individual term as $t\to\infty$, one also needs to ensure that the rest of the terms \emph{together} do not contribute to anything larger than $o(t^{\psi(c)})$. This follows either from the estimates of Proposition \ref{convgamma} or simply from the fact that we have absolute convergence at $\bf x$ and then all other terms decay faster.
Indeed, if we have an absolutely convergent series $\sum_{i\geq 0} a_i$ and then consider $\sum_{i\geq 0} a_i t^{\alpha_i}$ with $\alpha_0-\alpha_i$ larger than some positive $\varepsilon$, then as $t\to\infty$ we have
$$
\sum_{i\geq 0} a_i t^{\alpha_i} =a_0 t^{\alpha_0}(1+o(1))
$$
because
$$
\Big\vert \sum_{i>0} a_i t^{\alpha_i-\alpha_0} \Big\vert  \leq t^{-\varepsilon} \sum_{i>0} |a_i |.
$$
We can apply it to our situation  since $(l_i)$ are in a countable set and there exists $\varepsilon>0$ so that for all other terms the inequality  \eqref{inequality} is strict by at least $\varepsilon$. 
The logarithmic terms $\prod_i( t^{-\psi(v_1)} x_i)^{\frac{ D_i}{2\pi \ii}}$ can be absorbed by a slight change of $\varepsilon$.
\end{proof}

We can state a similar result for $\Gamma^\circ$. For $d\in C^\circ$ we consider the element of $L_{d,\gamma^\vee(d)}$
$$
-d=\sum_{i\in \sigma(d)} (-d_i) v_i.
$$
\begin{lemma}\label{lead-d}
As $t\to +\infty$, we have for $c\in C\cap N$ and $\gamma\neq  \gamma^\vee(d)$ the $\gamma$ summand of 
$\Gamma^\circ_d(t^{-\psi(v_1)} x_1, \ldots, t^{-\psi(v_n)}x_n)$ is $o(t^{\psi(d)})$. For $\gamma=\gamma^\vee(d)$ we have
\begin{align*}
\Gamma_d(t^{-\psi(v_1)} x_1, \ldots)
=
t^{\psi(d)} \prod_{i=1}^n {\rm e}^{\frac {D_i}{2\pi{\rm i}}( \log x_i - \psi(v_i) \log t)}
\prod_{i=1}^{n}\frac {x_i^{-d_i}}{\Gamma(1-d_i +\frac {D_i}{2\pi{\rm i}} )}\\
\left(\prod_{i\in\sigma(d)}D_i^{-1}\right)F_{\sigma(d)}
(1+o(1)).
\end{align*}
\end{lemma}

\begin{proof}
The proof is analogous to that of Lemma \ref{lead-c} and is left to the reader.
\end{proof}

Now we use this information about the asymptotic behavior of $\Gamma$ and $\Gamma^\circ$ to compute the constant
$ \langle\Gamma,\Gamma^\circ\rangle=\sum_{c,d,I}\xi_{c,d,I}\operatorname{Vol}_I\left(\prod_{i\in I}x_i\right) \Gamma_c\otimes\Gamma^\circ_d$
where $\xi$ are defined in Theorem \ref{main theorem 1}.

\smallskip
As in Section \ref{sec.pairing},
let $I$ be a subset of $\{1,\ldots,n\}$ of size ${\rm rk} N$, which may or may not be a cone in $\Sigma$. Let $c$ and $d$ be 
such that $c+d = \sum_{i\in I} v_i$ and $c+\varepsilon v, d-\varepsilon v \in \sum_{i\in I}\R_{\geq 0} v_i$ for small $\varepsilon >0$. 
The following observation is key. 
\begin{proposition}\label{most0}
Under the above assumptions on $c,d,I$ we have 
$$
\lim_{t\to +\infty} \prod_{i=1}^n (t^{-\psi(v_i)}x_i) \Gamma_c (t^{-\psi(v_1)} x_1, \ldots)\Gamma^\circ_d(t^{-\psi(v_1)} x_1, \ldots) =0
$$
unless $\gamma(d)=\gamma^\vee(c)$ and $I$ contains $\sigma(\gamma(c))$.  
\end{proposition}

\begin{proof}
Since $c$ and $d$ are contained in $\sum_{i\in I}\R_{\geq 0} v_i$ and $c+d = v_I$, we have
$$
c=\sum_{i\in I} \alpha_i v_i,~~d=\sum_{i\in I} (1-\alpha_i)v_i
$$
with $\alpha_i\in [0,1]$. Convexity of $\psi$ implies that 
\begin{align}\label{cd-bound}
\psi(c) \leq \sum_i \alpha_i \psi(v_i),~~\psi(d) \leq \sum_i (1- \alpha_i )\psi(v_i)
\end{align}
which leads to $\psi(c) + \psi(d) - \sum_i \psi(v_i)\leq 0$, so we can use Propositions \ref{lead-c} and \ref{lead-d} to see that the leading power of $t$ is nonpositive. In fact, it is negative, unless the inequalities in \eqref{cd-bound} are equalities, which means that the subset of $I$ for which $\alpha_i>0$ is a cone in $\Sigma$, and similarly for the subset of $\alpha_i<1$.  This implies the claim.
\end{proof}

\begin{proposition}\label{contributions}
If $\gamma(c)=\gamma^\vee, ~\gamma(d) = \gamma=\sum_{i\in I}\gamma_i v_i$, then we define $I_c$ to be the subset of $I$ such that the coefficients $c_i$ of $c$  are equal to $1$ and similarly for $I_d$.
The asymptotic behavior as $t\to \infty$ is 
\begin{align*}
&\prod_{i=1}^n (t^{-\psi(v_i)}x_i) \Gamma_c (t^{-\psi(v_1)} x_1, \ldots)\Gamma^\circ_d(t^{-\psi(v_1)} x_1, \ldots) 
= o(1) +\frac{1}{(2\pi \ii)^{\operatorname{rk}N-|\sigma(\gamma)|}}
\\
		&\cdot\frac{D_{I_c}}{\prod_{i\in\sigma(\gamma)}\Gamma(\gamma_i+\frac{D_i}{2\pi {\rm i}})\prod_{i\in\operatorname{Star}(\sigma(\gamma))\backslash\sigma(\gamma)}\Gamma(1+\frac{D_i}{2\pi {\rm i}})}
		\prod_{i=1}^n {\rm e}^{\frac {D_i}{2\pi{\rm i}}( \log x_i - \psi(v_i) \log t)}\\
		&\bigotimes \frac{F_{I_d}}{\prod_{i\in\sigma(\gamma)}\Gamma(1-\gamma_i+\frac{D_i}{2\pi {\rm i}})\prod_{i\in\operatorname{Star}(\sigma(\gamma))\backslash\sigma(\gamma)}\Gamma(1+\frac{D_i}{2\pi {\rm i}})}
		 \prod_{i=1}^n {\rm e}^{\frac {D_i}{2\pi{\rm i}}( \log x_i - \psi(v_i) \log t)}
		\end{align*}
in $H_{\gamma}\otimes H_{\gamma^{\vee}}^c$.
\end{proposition}

\begin{proof}
The proof of Proposition \ref{most0} shows that the only contribution other than $o(1)$ can come from the terms that give better than 
$o(t^{\psi(c)})$ and $o(t^{\psi(d)})$ contributions to the asymptotic behavior of $\Phi_c$ and $\Phi_d$. So by Propositions \ref{lead-c} and \ref{lead-d} the only contributions come from elements of $L_{c,\gamma^\vee}$ and $L_{d,\gamma}$ given by 
$$
-c=\sum_{i\in I} (-c_i)v_i,~~-d=\sum_{i\in I} (c_i-1) v_i.
$$
For $i\in \sigma(\gamma)$ we note that $\gamma_i=1-c_i$ if $c_i\in (0,1)$. For $i\in I_c$ we use
$$
\frac 1 {\Gamma(1-c_i +\frac{D_i}{2\pi{\rm i}})} = \frac 1 {\Gamma(\frac{D_i}{2\pi{\rm i}})} =
\frac {\frac {D_i}{2\pi{\rm i}}}{\Gamma(1 +\frac{D_i}{2\pi{\rm i}})} 
$$
and similarly for $i\in I_d$, and the result follows.
\end{proof}

Now we recall that $\langle \Gamma,\Gamma^\circ \rangle$ is constant.
\begin{corollary}\label{key}
The constant pairing $\langle \Gamma,\Gamma^\circ \rangle$ lies in $\bigoplus_\gamma H_\gamma \otimes H_{\gamma^\vee}^c$ and is given 
by 
\begin{align*}
		\frac{1}{(2\pi \ii)^{\operatorname{rk}N}}\bigoplus_{\gamma}\sum_{\substack{c\in C,d\in C^{\circ} \\ |I|=\operatorname{rk}N}}\xi_{c,d,I}\operatorname{Vol}_I(2\pi \ii)^{|\sigma(\gamma)|}\frac{D_{I_c}}{\widehat{\Gamma}_{\gamma}}\otimes\frac{F_{I_d}}{\widehat{\Gamma}_{\gamma^{\vee}}}
	\end{align*}
where $\widehat{\Gamma}_{\gamma} = \prod_{i\in\sigma(\gamma)}\Gamma(\gamma_i+\frac{D_i}{2\pi {\rm i}})\prod_{i\in\operatorname{Star}(\sigma(\gamma))\backslash\sigma(\gamma)}\Gamma(1+\frac{D_i}{2\pi {\rm i}})$ and similarly 
for  $\widehat{\Gamma}_{\gamma^\vee}$.
There also holds for each $k$
\begin{align*}
0=\bigoplus_{\gamma}\sum_{\substack{c\in C,d\in C^{\circ} \\ |I|=\operatorname{rk}N}}\xi_{c,d,I}\operatorname{Vol}_I(2\pi \ii)^{|\sigma(\gamma)|}\Big(D_k\frac{D_{I_c}}{\widehat{\Gamma}_{\gamma}}\Big)\otimes\frac{F_{I_d}}{\widehat{\Gamma}_{\gamma^{\vee}}}
\\
+\bigoplus_{\gamma}\sum_{\substack{c\in C,d\in C^{\circ} \\ |I|=\operatorname{rk}N}}\xi_{c,d,I}\operatorname{Vol}_I(2\pi \ii)^{|\sigma(\gamma)|}\frac{D_{I_c}}{\widehat{\Gamma}_{\gamma}}\otimes \Big( D_k\frac{F_{I_d}}{\widehat{\Gamma}_{\gamma^{\vee}}}\Big).
	\end{align*}
\end{corollary}

\begin{proof}
Proposition \ref{contributions} gives the asymptotic behavior of $\langle \Gamma,\Gamma^\circ \rangle$ as a polynomial in $\log x_i$. However, we also know it is a constant by Theorem \ref{main theorem 1}.  The first statement of the proposition is reading off the constant term of the polynomial and the second statement is reading off the coefficient by $\log x_k$.
\end{proof}

\section{Euler characteristic pairing}\label{sec.euler}

Now we are ready to prove that the pairing of Gamma series $\langle\Gamma,\Gamma^{\circ}\rangle$ is inverse to the Euler characteristic pairing on $\mathbb{P}_{\Sigma}$. 
Before we state the main theorem of this section, we have the following useful observation, which is an orbifold analog of the 
relationship between the $\Gamma$-class and the Todd class of a smooth manifold. Recall that $*$ is the duality map on $H$ defined in Proposition \ref{def-eulerpairing}.

\begin{lemma}\label{gamma class and todd class}
	$(\widehat{\Gamma}_{\gamma})^*\widehat{\Gamma}_{\gamma^{\vee}}=(2\pi \ii)^{|\sigma(\gamma)|}(-1)^{\deg\gamma^{\vee}}\operatorname{Td}(\gamma^{\vee})$.
\end{lemma}
\begin{proof}
	We can expand $(\widehat{\Gamma}_{\gamma})^*\widehat{\Gamma}_{\gamma^{\vee}}$ as
	\begin{align*}
		&\prod_{i\in\sigma(\gamma)}\Gamma(\gamma_i+\frac{D_i}{2\pi \ii})^* \Gamma(1-\gamma_i+\frac{D_i}{2\pi \ii})
		&\cdot\prod_{i\in\operatorname{Star}(\sigma(\gamma))\backslash\sigma(\gamma)}\Gamma(1+\frac{D_i}{2\pi \ii})^*\Gamma(1+\frac{D_i}{2\pi \ii})\\
		=&\prod_{i\in\sigma(\gamma)}\Gamma(\gamma_i-\frac{D_i}{2\pi \ii}) \Gamma(1-\gamma_i+\frac{D_i}{2\pi \ii})
		&\cdot\prod_{i\in\operatorname{Star}(\sigma(\gamma))\backslash\sigma(\gamma)}\Gamma(1-\frac{D_i}{2\pi \ii})\Gamma(1+\frac{D_i}{2\pi \ii}).
	\end{align*}
We use the identity $\Gamma(z)\Gamma(1-z)
	=-\frac{2\pi \ii \,\ee^{\pi\ii z}}{1-\ee^{2\pi\ii z}}$ to rewrite the first product as
\begin{align*}
	&(-2\pi\ii)^{|\sigma(\gamma)| }\ee^{\sum_{i\in \sigma(\gamma) }  \pi \ii \gamma_i} \ee^{-\frac 12\sum_{i\in \sigma(\gamma) } D_i} 
	\prod_{_{i\in\sigma(\gamma)}}\frac 1{1-\ee^{2\pi\ii \gamma_i-D_i}}.
	\end{align*}
For the second product, we use 
$
\Gamma(1-\frac{z}{2\pi \ii})\Gamma(1+\frac{z}{2\pi \ii})=\frac{z\ee^{-\frac{z}{2}}}{1-\ee^{-z}}
$
to rewrite it as
\begin{align*}
	\ee^{-\frac{1}{2}\sum_{i\in\operatorname{Star}(\sigma(\gamma))\backslash\sigma(\gamma)} D_i}\prod_{i\in\operatorname{Star}(\sigma(\gamma))\backslash\sigma(\gamma)}\frac{D_i}{1-\ee^{-D_i}}.
\end{align*}
Putting the two formulas together, we get
\begin{align*}
&(\widehat{\Gamma}_{\gamma})^*\widehat{\Gamma}_{\gamma^{\vee}}=(2\pi \ii)^{|\sigma(\gamma)|}(-1)^{\deg\gamma^{\vee}}\ee^{-\frac{1}{2}\sum_{i\in\operatorname{Star}(\sigma(\gamma))}D_i}
\frac {\prod_{i\in\operatorname{Star}(\sigma(\gamma))\backslash\sigma(\gamma)}D_i }
{ \prod_{i\in\operatorname{Star}(\sigma(\gamma))}
{1-\ee^{-D_i}}}
\\
&=(2\pi \ii)^{|\sigma(\gamma)|}(-1)^{\deg\gamma^{\vee}}\operatorname{Td}(\gamma^{\vee})
\end{align*}
where we used $\sum_{i\in\operatorname{Star}(\sigma(\gamma))}D_i=\sum_{i=1}^n D_i=0$.
\end{proof}

Now we can state and prove the main theorem of this section. Recall that we defined the pairing $\langle\cdot,\cdot\rangle$ 
on solutions of the better-behaved GKZ systems. When we apply it to $\Gamma$ and $\Gamma^\circ$, we get a constant element of $H\otimes H^c$.
\begin{theorem}\label{main theorem 2}
The constant pairing $\langle \Gamma,\Gamma^\circ \rangle$ is equal up to a constant factor to the inverse of the Euler characteristic pairing $\chi(-,-):H\otimes H^c\rightarrow\C$.
\end{theorem}

\begin{proof}
	It's clear that we can consider each twisted sector individually. For a fixed $\gamma$, the statement is equivalent to the assertion that 
\begin{align*}
	\sum_{\substack{c\in C,d\in C^{\circ} \\ |I|=\operatorname{rk}N}}\xi_{c,d,I}\operatorname{Vol}_I(2\pi \ii)^{|\sigma(\gamma)|}\chi\left(P,\frac{F_{I_d}}{\widehat{\Gamma}_{\gamma^{\vee}}}\right)\frac{D_{I_c}}{\widehat{\Gamma}_{\gamma}}=P
\end{align*}
holds for all classes $P\in H_{\gamma}$. Since the class $\widehat{\Gamma}_{\gamma}$ is invertible in $H_{\gamma}$, dividing by it induces an automorphism on the cohomology, hence it suffices to prove
\begin{align}\label{inductionstatement}
	\sum_{\substack{c\in C,d\in C^{\circ} \\ |I|=\operatorname{rk}N}}\xi_{c,d,I}\operatorname{Vol}_I(2\pi \ii)^{|\sigma(\gamma)|}\chi\left(\frac{P}{\widehat{\Gamma}_{\gamma}},\frac{F_{I_d}}{\widehat{\Gamma}_{\gamma^{\vee}}}\right)\frac{D_{I_c}}{\widehat{\Gamma}_{\gamma}}=\frac{P}{\widehat{\Gamma}_{\gamma}}
\end{align}
for all $P$. We prove this by induction on the degree of $P$.

\smallskip
The base case $\deg{P}=0$ corresponds to $P=1_{\gamma}$. Since
$$
\chi\left(\frac{1_\gamma}{\widehat{\Gamma}_{\gamma}},\frac{F_{I_d}}{\widehat{\Gamma}_{\gamma^{\vee}}}\right)=0$$
unless $|I_d|=\rk N - |\sigma(\gamma)|$, the equation becomes
\begin{align}\label{base case}
	\sum_{|I_d|=\operatorname{rk}N-|\sigma(\gamma)|}\xi_{\gamma^{\vee},\gamma+v_{I_d},I_d\sqcup\sigma(\gamma)}\operatorname{Vol}_{I_d\sqcup\sigma(\gamma)}(2\pi \ii)^{|\sigma(\gamma)|}\chi\left(\frac{1_\gamma}{\widehat{\Gamma}_{\gamma}},\frac{F_{I_d}}{\widehat{\Gamma}_{\gamma^{\vee}}}\right)\frac{1_\gamma}{\widehat{\Gamma}_{\gamma}}=\frac{1_\gamma}{\widehat{\Gamma}_{\gamma}}.
\end{align}
Then by definition of $\chi$ and Lemma \ref{gamma class and todd class}, we have
\begin{align*}
	\chi\left(\frac{1_\gamma}{\widehat{\Gamma}_{\gamma}},\frac{F_{I_d}}{\widehat{\Gamma}_{\gamma^{\vee}}}\right)&=\frac{1}{|\operatorname{Box}(\sigma(\gamma))|}\int_{\gamma^{\vee}}\operatorname{Td}(\gamma^{\vee})\left(\frac{1}{\widehat{\Gamma}_{\gamma}}\right)^*\frac{F_{I_d}}{\widehat{\Gamma}_{\gamma^{\vee}}}\\
	&=\frac{1}{|\operatorname{Box}(\sigma(\gamma))|}\int_{\gamma^{\vee}}\frac{F_{I_d}}{(\widehat{\Gamma}_{\gamma})^*\widehat{\Gamma}_{\gamma^{\vee}}}\operatorname{Td}(\gamma^{\vee})\\
	&=\frac{1}{|\operatorname{Box}(\sigma(\gamma))|}\int_{\gamma^{\vee}}\frac{F_{I_d}}{(2\pi \ii)^{|\sigma(\gamma)|}(-1)^{\deg\gamma^{\vee}}}\\
	&=\frac{(-1)^{\deg\gamma^{\vee}}}{(2\pi \ii)^{|\sigma(\gamma)|}\operatorname{Vol}_{\overline{I_d}}|\operatorname{Box}(\sigma(\gamma))|}
\end{align*}
here $\operatorname{Vol}_{\overline{I_d}}$ denotes the volume of the cone $\overline{\sigma_{I_d}}$ in the quotient fan $\Sigma/\sigma(\gamma)$. Note that we have
\begin{align*}
	\operatorname{Vol}_{I_d\sqcup\sigma(\gamma)}=\operatorname{Vol}_{\overline{I_d}}|\operatorname{Box}(\sigma(\gamma))|
\end{align*}
hence \eqref{base case} becomes
\begin{align*}
	\sum_{|I_d|=\operatorname{rk}N-|\sigma(\gamma)|}(-1)^{\deg\gamma^{\vee}}\xi_{\gamma^{\vee},\gamma+v_{I_d},I_d\sqcup\sigma(\gamma)}=1.
\end{align*}
If we perturb $\gamma^{\vee}$ by $\varepsilon v$, then it will fall in the interior of exactly one maximal cone in $\Sigma$, and the corresponding coefficient $\xi$ is the only nonzero term in the sum above (recall the definition of $\xi_{c,d,I}$ in Theorem \ref{main theorem 1}), which is equal to
\begin{align*}
	(-1)^{\deg\gamma^{\vee}}(-1)^{\deg{\gamma^{\vee}}}=1
\end{align*}
So the base case is proved.

\smallskip
Now we assume the equality \eqref{inductionstatement} holds for all classes of degree less than $m$. Since the cohomology $H_{\gamma}$ is generated as an algebra by classes $D_k$, it suffices to prove the identity 
\begin{align*}
	\sum_{\substack{c\in C,d\in C^{\circ} \\ |I|=\operatorname{rk}N}}\xi_{c,d,I}\operatorname{Vol}_I(2\pi \ii)^{|\sigma(\gamma)|}\chi\left(\frac{D_k P}{\widehat{\Gamma}_{\gamma}},\frac{F_{I_d}}{\widehat{\Gamma}_{\gamma^{\vee}}}\right)D_{I_c}=D_kP
\end{align*}
for each $D_k P$ where $P\in H_\gamma$ is of degree $m-1$. Since $D_k$ is skew-symmetric with respect to the $\chi$ pairing, the above statement can be rewritten as 
\begin{align*}
	D_k P = - \sum_{\substack{c\in C,d\in C^{\circ} \\ |I|=\operatorname{rk}N}}\xi_{c,d,I}\operatorname{Vol}_I(2\pi \ii)^{|\sigma(\gamma)|}\chi\left(\frac{ P}{\widehat{\Gamma}_{\gamma}},\frac{ D_kF_{I_d}}{\widehat{\Gamma}_{\gamma^{\vee}}}\right)D_{I_c}.
\end{align*}
On the other hand, we can multiply the induction assumption for $P$ by $D_k$ to get 
\begin{align*}
	\sum_{\substack{c\in C,d\in C^{\circ} \\ |I|=\operatorname{rk}N}}\xi_{c,d,I}\operatorname{Vol}_I(2\pi \ii)^{|\sigma(\gamma)|}\chi\left(\frac{P}{\widehat{\Gamma}_{\gamma}},\frac{F_{I_d}}{\widehat{\Gamma}_{\gamma^{\vee}}}\right)D_k \,D_{I_c}= D_k P.
\end{align*}
Compare these two identities. It suffices to show
\begin{equation}\label{induction}
\begin{split}
	0=\sum_{\substack{c\in C,d\in C^{\circ} \\ |I|=\operatorname{rk}N}}\xi_{c,d,I}&\operatorname{Vol}_I\left(D_k\cdot\frac{D_{I_c}}{\widehat{\Gamma}_{\gamma}}\right)\otimes\frac{F_{I_d}}{\widehat{\Gamma}_{\gamma^{\vee}}}\\
	&+\sum_{\substack{c\in C,d\in C^{\circ} \\ |I|=\operatorname{rk}N}}\xi_{c,d,I}\operatorname{Vol}_I\frac{D_{I_c}}{\widehat{\Gamma}_{\gamma}}\otimes\left(D_k\cdot\frac{F_{I_d}}{\widehat{\Gamma}_{\gamma^{\vee}}}\right)
\end{split}
\end{equation}
which follows from Corollary \ref{key}.
\end{proof}

\begin{remark}
Theorem \ref{main theorem 2} implies, in particular, that the pairing of Theorem \ref{main theorem 1} is nondegenerate and is independent of $v$. We are not aware of a direct proof of this fact.
\end{remark}

We conclude this section by an explanation of our motivation behind the definition of the coefficients $\xi_{c,d,I}$ in Theorem \ref{main theorem 1}. This definition is inspired by the following fan displacement resolution of diagonal formula of Fulton-Sturmfels \cite{FS}.
\begin{proposition}
	Let $X$ be the toric variety corresponds to a \textit{complete} fan $\Sigma$ in a lattice $N$, denote the diagonal embedding $X\hookrightarrow X\times X$ by $\delta$. Let $\sigma\in\Sigma$ be any cone and $v$ a generic point in $N$, then the diagonal class decomposes as
	\begin{align*}
		[\delta(V(\sigma))]=\sum_{\sigma_1,\sigma_2} m_{\sigma_1,\sigma_2}^{\sigma}\cdot[V(\tau_1)\times V(\tau_2)]
	\end{align*}
	where $m_{\sigma_1,\sigma_2}^{\sigma}=[N:N_{\sigma_1}+N_{\sigma_2}]$ and the sum is over all cones $\sigma_1,\sigma_2\in\Sigma$ with $\operatorname{codim}\sigma_1+\operatorname{codim}\sigma_2=\operatorname{codim}\sigma$ and $\sigma\subseteq\sigma_1,\sigma_2$ such that $(v+\sigma_1)\cap\sigma_2\not=\emptyset$.
\end{proposition}

Note that the coefficient $m_{\sigma_1,\sigma_2}^{\sigma}$ is exactly the volume $\operatorname{Vol}_{\sigma_1\cup\sigma_2}$ of the cone spanned by $\sigma_1$ and $\sigma_2$. This formula cannot be applied to our case directly, since the toric varieties they worked with are complete while ours are not. Nevertheless we have the following relationship between the definition of $\xi_{c,d,I}$ and the conditions occurred in Fulton-Sturmfels formula.

\begin{proposition}
Let $c,d\in\sigma_I$ and $v$ be a generic point in $C^{\circ}$. Then both $c+\varepsilon v$ and $d-\varepsilon v$ lies in $\sigma_I^{\circ}$ for all sufficiently small $\varepsilon>0$ if and only if
\begin{align*}
    (v+\sigma(c))\cap\sigma(d)\not=\emptyset
\end{align*}
where $\sigma(c)$ denotes the minimal cone of $\Sigma$ that contains $c$.
\end{proposition}
\begin{proof}
 Assume both $c+\varepsilon v$ and $d-\varepsilon v$ lies in $\sigma_I^{\circ}$. Then we can write $c+\varepsilon v=\sum_{i\in I} s_i v_i$ where all $s_i\in(0,1)$. Recall that $I=\sigma(c)\cup\sigma(d)=I_c\sqcup I_d\sqcup\sigma(\gamma(c))$, this equation can be rewritten into the form $v=v_1-v_2$, where $v_1\in\sigma(c)$ and $v_2\in\sigma(d)$, which is equivalent to the second statement. The other direction can be proved similarly.
\end{proof}

\begin{remark}
We believe our methods should allow one to give a new proof of the Fulton-Sturmfels formula, which could be done by restricting our results to the twisted sectors that are compact. We do not go into details further in this paper.
\end{remark}

\section{Extensions and open questions}\label{sec.generic}
There is a more general version of the better-behaved GKZ systems which includes a parameter $\beta\in N_\C$, 
with $\beta=0$ case being the one we considered so far. Namely, the torus homogeneity equations of Definition \ref{defGKZ} read
$$
\sum_{i=1}^n\langle\mu,v_i\rangle x_i\partial_i\Phi_c+\langle\mu,c-\beta\rangle\Phi_c=0
$$
and similarly for $\Psi_d$. Much of what we did in this paper is applicable to the pair of better behaved GKZ systems with parameters $\pm\beta$. For instance, we readily observe that our argument in Section \ref{sec.pairing} goes through for arbitrary parameter $\beta$ to give a pairing between spaces of solutions to $\operatorname{bbGKZ}(C,\beta)$ and $\operatorname{bbGKZ}(C^{\circ},-\beta)$.

\smallskip
We would like to see what happens in the limit given by a regular subdivision $\Sigma$ for a generic $\beta$.
While there are certain versions of $H$ and $H^c$ considered in \cite{horja2013toric} it will be easier for our purposes to simply write ${\rm Vol}(\Delta)$ linearly independent solutions given by $\Gamma$-series, essentially along the lines of the solutions of the original GKZ paper \cite{GKZoriginal}.

\smallskip
Let $\Sigma$ be a regular subdivision of $C$. For each maximum-dimensional cone $\sigma$ we consider ${\rm Vol}(\sigma)$ linearly independent solutions in the large K\"ahler limit of $\PP_\Sigma$, in bijection with the elements $\gamma$ of $N/\sum_{i\in\sigma}\Z v_i$. Namely, we 
define the set $L_{c,\gamma,\sigma;\beta}\subset \C^n$ by
$$
\sum_{i=1}^n l_i v_i = \beta- c
$$
and the properties $l_i\in \Z$ for all $i\not\in\sigma$ and 
$c+\sum_{i\not\in \sigma} l_i v_i  = -\gamma\hskip -3pt \mod \sum_{i\in\sigma}\Z v_i.$
Then for each $\gamma$ we define a solution $\Phi^{\gamma,\sigma}$ of $\operatorname{bbGKZ}(C,\beta)$ by 
$$
\Phi^{\gamma,\sigma}_c(x_1,\ldots,x_n) = \sum_{l\in L_{c,\gamma,\sigma;\beta} } \prod_{i=1}^n \frac{x_i^{l_i}}{\Gamma(1+l_i)}.
$$ 
We define $\Gamma$-series solutions $\Psi^{\gamma,\sigma}$ to $\operatorname{bbGKZ}(C^{\circ},-\beta)$ in the same way
by
$$
\Psi^{\gamma,\sigma}_d(x_1,\ldots,x_n) = \sum_{l\in L_{d,\gamma,\sigma;-\beta} } \prod_{i=1}^n \frac{x_i^{l_i}}{\Gamma(1+l_i)}.
$$
Note that in the case of generic $\beta$ every solution of  $\operatorname{bbGKZ}(C^{\circ},-\beta)$ 
can be uniquely extended to solutions of $\operatorname{bbGKZ}(C,-\beta)$.
It is not hard to show that these $\Phi_c$  and $\Psi_d$ converge uniformly on compacts in the region \eqref{region} for an appropriate choice of $\hat\psi$. Moreover, as $\sigma$ and $\gamma$ vary, we get bases of the space of solutions, with linear independence assured by them lying in different eigenspaces of the monodromy operators for small loops around $x_i=0$. 

\smallskip
Monodromy considerations imply that for the pairing $\langle\cdot,\cdot\rangle$ of Section \ref{sec.pairing} we have
$
\langle \Phi^{\gamma,\sigma},\Psi^{\gamma',\sigma'}\rangle=0
$
unless $\sigma= \sigma'$ and $\gamma=-\gamma'\hskip -3pt\mod \sum_{i\in\sigma}\Z v_i$. In the latter case, 
the constant contribution will happen for $l_i+l_i'=0$ for $i\not\in I$ and $l_i+l_i' = -1$ for $i\in I$. If any of $l_i, l_i'$ is a negative integer, then the corresponding term vanishes, due to a pole of $\Gamma$,
so we may assume that they are nonnegative for $i\not\in \sigma$, which then implies that 
$$I = \sigma;~l_i+l_i'=-1, {\rm~for~} i\in\sigma; ~l_i=l_i'=0 {\rm~for~}i\not\in\sigma.$$
This implies that 
$c=-\gamma\hskip -3pt \mod \sum_{i\in\sigma}\Z v_i$ and 
$d=\gamma\hskip -3pt \mod \sum_{i\in\sigma}\Z v_i$.

\smallskip
We claim that for any $\gamma$ there exists exactly one  pair $(c,d)$ in $\sigma$ satisfying this constraint and  $\xi_{c,d,\sigma}\not=0$. 
The definition of the coefficients $\xi$ of the pairing implies that we must also have  
$c+d = \sum_{i\in \sigma} v_i$ with $c+\varepsilon v$ and $d-\varepsilon v$ in the corresponding cone $\sum_{i\in \sigma} \R_{\geq 0}v_i$ for all small $\varepsilon>0$. 
We can write $\beta$, $v$ and $\gamma$ uniquely as
$$
\beta = \sum_{i\in \sigma}\beta_i v_i,~v= \sum_{i\in \sigma}s_i v_i,~\gamma = \sum_i \gamma_i v_i 
$$
with $\gamma_i \in [0,1)$. It is then easy to see that $\xi_{c,d,\sigma}$ is nonzero if and only if
\begin{align*}
	c&=\sum_{\{i:\gamma_i\neq 0\}}(1-\gamma_i)v_i + \sum_{\{i:\gamma_i=0,s_i<0\}}v_i,\\
	d&=\sum_{\{i:\gamma_i\neq 0\}}\gamma_iv_i + \sum_{\{i:\gamma_i=0,s_i>0\}}v_i.
\end{align*}
Thus for $\gamma_i\neq 0$ we have 
$
l_i = \beta_i -1+ \gamma_i, ~l_i' = -\beta_i -\gamma_i.
$
For $\gamma_i= 0$ and $s_i>0$ we have 
$
l_i = \beta_i,~ l_i' = -1-\beta_i
$
and for  $\gamma_i= 0$ and $s_i<0$ we have 
$
l_i = -1+\beta_i,~ l_i' = -\beta_i.
$
In particular, 
$$\deg(c) 
=-\deg (\gamma) + \rk N - \#\{i:\gamma_i=0,s_i>0\} .$$

\smallskip
Therefore the pairing is given by
\begin{align*}
	\langle \Phi^{\gamma,\sigma},\Psi^{-\gamma,\sigma}\rangle&=
	(-1)^{\deg (c)}\,{\rm Vol(\sigma)}\prod_{\gamma_i\not=0}\frac 1{\Gamma(\beta_i+\gamma_i)\Gamma(1-\beta_i-\gamma_i)}
	\\
&	\prod_{\gamma_i=0,s_i>0}\frac 1{\Gamma(1+\beta_i)\Gamma(-\beta_i)}
	\prod_{\gamma_i=0,s_i<0}\frac 1{\Gamma(\beta_i)\Gamma(1-\beta_i)}
	\\
&=(-1)^{\deg (c)}\,{\rm Vol(\sigma)}\prod_{\gamma_i\not=0}\frac{\ee^{2\pi{\rm i}(\beta_i+\gamma_i)}-1}{2\pi{\rm i}\,\ee^{\pi{\rm i}(\beta_i+\gamma_i)}}
 \\
&\prod_{\gamma_i=0,s_i>0}\frac{\ee^{2\pi{\rm i}(\beta_i+1)}-1}{2\pi{\rm i}\,\ee^{\pi{\rm i}(\beta_i+1)}}
\prod_{\gamma_i=0,s_i<0}\frac{\ee^{2\pi{\rm i}\beta_i}-1}{2\pi{\rm i}\,\ee^{\pi{\rm i}\beta_i}}
\\
&=\frac{(-1)^{\deg (c)}
\,{\rm Vol(\sigma)}}{(2\pi{\rm i})^{\rk N}}\ee^{-\pi{\rm i}\sum_{i\in\sigma}(\beta_i+\gamma_i)}\prod_{\gamma_i=0,s_i>0}\ee^{\pi{\rm i}}\prod_{i\in\sigma}(\ee^{2\pi{\rm i}(\beta_i+\gamma_i)}-1)
\\
&=\frac{{\rm Vol(\sigma)}}{(2\pi{\rm i})^{\rk N}}\ee^{-\pi{\rm i}\deg (\beta) - 2\pi \ii \deg (\gamma)}\prod_{i\in\sigma}(1-\ee^{2\pi{\rm i}(\beta_i+\gamma_i)})\\
&=\frac{e^{-\pi{\rm i}\deg(\beta)}{\rm Vol(\sigma)}}{(2\pi{\rm i})^{\rk N}}\prod_{i\in\sigma}(1-e^{2\pi{\rm i}(\beta_i+\gamma_i)}).
\end{align*}
\begin{remark}
An immediate consequence of the above calculation is that the pairing $\langle\cdot,\cdot\rangle$  is non-degenerate for a generic $\beta$.
\end{remark}

\medskip
{\bf Further directions.}
We conclude this section by stating some open problems related to our construction, in no particular order. 

\begin{itemize}
\item Is the pairing of this paper nondegenerate for all $\beta$? We know this to be the case for $\beta=0$ and $\beta$ generic, and it seems likely to be always true.

\item
We would like to settle the analytic continuation conjecture of \cite{BHconj} to extend the main result 
of \cite{MellinBarnes} to the better-behaved GKZ systems. One consequence of Theorem \ref{main theorem 2} is that it should be enough to just work with the usual $K$-theory and the compactly supported version should follow from duality.

\item
What is the HMS counterpart of our pairing from the point of view of Fukaya-Seidel categories for the mirror potential? Our formula for the pairing is quite simple, so presumably so should be the mirror version of it. We refer to \cite{Fukaya}, \cite{Seidel} for background.

\item
Solutions to bbGKZ systems come with a lattice structure inherited from the $K$-theory of $\PP_\Sigma$ (it is independent of $\Sigma$). Can this structure be locally defined outside of the region of convergence of any $\Gamma$-series?

\end{itemize}


\begin{thebibliography}{99}
\bibitem{BHconj}
 L. Borisov, R.P. Horja, \emph{Applications of homological mirror symmetry to hypergeometric systems: duality conjectures.} Advances in Mathematics 271 (2015): 153--187.
 
 \bibitem{BHW}
L. Borisov, Z. Han, C. Wang, \emph{On duality of certain GKZ hypergeometric systems.} Asian Journal of Mathematics 25.1 (2021): 65--88.

 \bibitem{BH}
 L. Borisov, R.P. Horja, \emph{On the better behaved version of the GKZ hypergeometric system.} Mathematische Annalen 357.2 (2013): 585--603.

 \bibitem{MellinBarnes}
 L. Borisov, R.P. Horja, \emph{Mellin-Barnes integrals as Fourier-Mukai transforms.} Advances in Mathematics 207.2 (2006): 876--927.
 
\bibitem{Fukaya}
K. Fukaya, Y.-G. Oh, H. Ohta, K. Ono, \emph{Lagrangian intersection Floer theory: anomaly and obstruction}, Part II. Vol. 2. American Mathematical Soc., 2010.
 
 \bibitem{FS}
 W. Fulton, B. Sturmfels, \emph{Intersection theory on toric varieties.} Topology 36 (1997), no. 2, 335--353. 
 
 \bibitem{GKZoriginal}
 I. Gelfand, M. Kapranov, A. Zelevinsky, \emph{Generalized Euler integrals and A--hypergeometric functions.} Advances in Mathematics 84.2 (1990): 255--271.
 
 \bibitem{horja2013toric}
 R.P. Horja, \emph{Toric Deligne-Mumford stacks and the better behaved version of the GKZ hypergeometric system.} Strings, Gauge Fields, and the Geometry Behind: The Legacy of Maximilian Kreuzer. 2013. 329--348.
 
 \bibitem{HTT}
R. Hotta, K. Takeuchi, T. Tanisaki, \emph{D-modules, perverse sheaves, and representation theory}. Vol. 236. Springer Science 
\& Business Media, 2007.
 
 \bibitem{MMW}
 L. Matusevich, E. Miller, U. Walther, \emph{Homological methods for hypergeometric families}. Journal of the American Mathematical Society 18.4 (2005): 919--941.
 
\bibitem{Seidel}
 P. Seidel, \emph{Fukaya categories and Picard-Lefschetz theory}. Vol. 10. European Mathematical Society, 2008.
 

 
 
\end{thebibliography}
\end{document}